\documentclass[12pt]{amsart}

\RequirePackage[colorlinks,citecolor=blue,urlcolor=blue]{hyperref}

\usepackage{amsmath,amstext,amssymb,amsopn,amsthm}
\usepackage{url,verbatim}
\usepackage{mathtools}
\usepackage{enumerate}

\usepackage{color,graphicx}
\usepackage{subfig}

\usepackage[margin=30mm]{geometry}
\usepackage{eucal,mathrsfs,dsfont}

\allowdisplaybreaks

\newtheorem{theorem}{Theorem}[section]

\theoremstyle{definition}

\newtheorem{remark}[theorem]{Remark}

\numberwithin{equation}{section}

\newcommand{\calF}{\mathcal{F}}

\newcommand{\calG}{\mathcal{G}}
\newcommand{\calA}{\mathcal{A}}

\newcommand{\calW}{\mathcal{W}}

\newcommand{\calB}{\mathcal{B}}

\newcommand{\calH}{\mathcal{H}}

\newcommand{\R}{\mathds{R}}

\newcommand{\wh}{\widehat}
\newcommand{\wt}{\widetilde}

\DeclareMathOperator*{\argmin}{arg\,min}

\def \bv {{\bf v}}

\def \bz {{\bf z}}
\def\bx{{\bf x}}

\title[Billiard balls]{An improved upper bound on the number \\ of  billiard ball collisions}
\author{Krzysztof Burdzy}

\address{Department of Mathematics, Box 354350, University of Washington, Seattle, WA 98195}
\email{burdzy@uw.edu}

\thanks{Supported in part by Simons Foundation Grant 506732. }

\begin{document}

\begin{abstract}

We give a new upper bound $K_+$ on the number of totally elastic collisions of $n$ hard spheres with equal radii and equal masses in $R^d$. Our bound  satisfies $\log K_+ \leq c(d) n \log n$.

\end{abstract}

\maketitle

\section{Introduction}
\label{se:intro}

Consider a family of $n$ billiard balls in $\R^d$ reflecting from each other in a totally elastic way. We assume that their masses and radii are identical. Note that the ``billiard table'' has no walls---it is the whole space $\R^d$. We will prove the following upper bound for the number of collisions.

\begin{theorem}\label{j18.20}
The  number of collisions is bounded above by
\begin{align}\label{j19.2}
1600 \left( 1000 \cdot 32^{5^d}\right)^n  n^{((3/2)5^d+9/2) n +3/2} .
\end{align}
\end{theorem}

We will review the history of the problem in Section \ref{review}. Here we will discuss the question of optimality of our bound. Let $K_+'$ denote the best previously known upper bound for the number of collisions, stated below in  
\eqref{s26.1}. Let $K_-$ denote the best known lower bound in dimensions $d\geq 3$, stated below in \eqref{j19.1}. Let $K_+$ be the new upper bound given in \eqref{j19.2}. For $d\geq 3$, some constants $c_1, c_2$ and $c_3$ depending on  $d$, and large $n$,
\begin{align*}
 c_1 n \leq \log K_- < \log K_+\leq  c_2 n \log n < c_3 n^2 \log n \leq \log K_+'. 
\end{align*}
This shows that while there still remains a gap between the best lower and upper bounds $K_-$ and $K_+$, the gap is much smaller than between the previously known best upper bound $K_+'$ and the best known lower bound $K_-$.

The proof of Theorem \ref{j18.20} is based in an essential way on two results---one from \cite{BFK1} and another one from \cite{BuD18b}. The latter one shows that the family of all balls ``quickly'' splits into two non-interacting families. An estimate from \cite{BFK1} can be used to give an upper bound for the number of collisions on the initial (``short'') time interval. The  proof of  Theorem \ref{j18.20} is an inductive construction of a branching collection of ball subfamilies. Subfamilies that are leaves in the branching structure have ``lifetimes''  short enough so that the estimate mentioned above can be applied. Finally, we estimate the number of subfamilies in the branching structure. 

There are two sources of the factor of the form $n^{cn}$ in \eqref{j19.2}. One of them is a bound adopted from \cite{BFK1}. The other one is a branching construction in the present article. Hence, there is little hope for a significant improvement of the upper bound by fine-tuning the argument given in this paper. 

We will review the history of the problem in Section \ref{review}. The notation and assumptions will be presented in Section \ref{se:notation}.
The proof of Theorem \ref{j18.20} will be given in Section \ref{se:proof}.

\section{Hard ball collisions---historical review}
\label{review}

The question of whether a finite system of hard balls can have an infinite number of elastic collisions was posed by Ya.~Sinai. It was answered in negative in \cite{Vaser79}. For alternative proofs see \cite{Illner89, Illner90,IllnerChen,BuD18b}. 
It was proved in \cite{BFK1} that a system of $n$ balls in the Euclidean space undergoing elastic collisions can experience at most
\begin{align}\label{s26.1}
\left( 32 \sqrt{\frac{m_{\text{max}}}{m_{\text{min}}} } 
\frac{r_{\text{max}}}{r_{\text{min}}} n^{3/2}\right)^{n^2}
\end{align}
collisions. Here $m_{\text{max}}$ and $m_{\text{min}}$ denote the maximum and the minimum masses of the balls. Likewise, $r_{\text{max}}$ and $r_{\text{min}}$ denote the maximum and the minimum radii of the balls.
The following alternative upper bound for the maximum number of collisions appeared in \cite{BFK5}
\begin{align}\label{s26.2}
\left( 400 \frac{m_{\text{max}}}{m_{\text{min}}} 
 n^2\right)^{2n^4}.
\end{align}
  The papers \cite{BFK1,BFK2, BFK3,BFK4, BFK5} were the first to present universal bounds \eqref{s26.1}-\eqref{s26.2} on the number of collisions of $n$ hard balls in any dimension. No improved universal upper bounds were found since then, as far as we know.

It has been proved in \cite{BD} by example that the  number of elastic collisions of $n$ balls
in $d$-dimensional space is greater than $n^3/27$ for $n\geq 3$ and $d\geq 2$, for some initial conditions. The previously known lower bound was of order $n^2$ (that bound was for balls in dimension 1 and was totally elementary). The lower bound estimate was improved  in \cite{BuI18} to 
\begin{align}\label{j19.1}
2^{\lfloor n/2\rfloor}
\end{align}
 in dimensions $d\geq 3$.  

In a somewhat different direction, it has been shown in \cite{Serre} that no more than $O(n^2)$  collisions change the velocities of balls in a significant way.
 
\section{Assumptions and notation}
\label{se:notation}

We will consider $n\geq 3$ hard balls in $\R^d$, for $d\geq 2$, colliding elastically, on the time interval $(-\infty, \infty)$.
We will assume that the balls have equal masses and their radii are 1.

The center and velocity of the $k$-th ball will be denoted $x^k(t)$ and $v^k(t)$, for $k=1,2,\dots,n$. We will say that the $j$-th and $k$-th balls collide at time $t$ if $|x^j(t) - x^k(t)| = 2$ and their velocities change at this time.
The velocities are constant between collision times. 
We will write $\bx(t) = (x^1(t), \dots, x^n(t)) \in \R^{dn}$
and
$\bv(t) = (v^1(t), \dots, v^n(t))\in\R^{dn}$. Note that $\bv(t)$ is well defined only when $t$ is not a collision time, but both $\bv(t-)$ and $\bv(t+)$ are well defined for all times. 

Recall that all balls have the same mass. This implies that the velocities change at the moment of collision as follows. Suppose that the $j$-th and $k$-th balls collide at time $t$. 
This implies that the velocities $v^j(t-) $ and $ v^k(t-)$ (i.e., the velocities just before the collision) satisfy
\begin{align}\label{oc4.3}
(v^j(t-) - v^k(t-)) \cdot (x^j(t) - x^k(t)) < 0.
\end{align}
Let $x^{jk}(t) = (x^j(t) - x^k(t))/|x^j(t) - x^k(t)|$.
Then the velocities just after the collision are given by
\begin{align}\label{oc2.3}
v^j(t+) &= v^j(t-) + (v^k(t-) \cdot x^{jk}(t)) x^{jk}(t)
- (v^j(t-) \cdot x^{jk}(t)) x^{jk}(t),\\
v^k(t+) &= v^k(t-) + (v^j(t-) \cdot x^{jk}(t)) x^{jk}(t)
- (v^k(t-) \cdot x^{jk}(t)) x^{jk}(t).\label{oc4.2}
\end{align}
In other words, the balls exchange the components of their velocities that are parallel to the line through their centers at the moment of impact. The orthogonal components of velocities remain unchanged. 

Consider the following assumptions.

(A1)
The balls have equal masses and all radii are equal to 1.

(A2)
We will assume that there are no simultaneous collisions. 
 It is known that the set of vectors in the phase space of positions and velocities that lead to simultaneous collisions has  measure zero (see \cite{Alex76}). It has been proved in \cite[Thm.~4]{IllnerChen} that there are no accumulation points for collision times.

(A3)  We will assume  that the momentum of the system is zero, i.e., $\sum_{j=1}^n v^j(t) =0$ for all $t$. We can make this assumption because the number of collisions is the same in all inertial frames of reference. Since the total momentum is zero, the center of mass of all balls is constant, so it can  be assumed to be at the origin. This, together with the fact that all masses are equal, implies that $\sum_{j=1}^n x^j(t) =0$.  

(A4) We will assume without loss of generality that the total  ``energy'' is equal to 1, i.e., $|\bv(t+)|^2 =1$ for all $t$.  If the initial energy is not zero then we can multiply all velocity vectors  by the same scalar constant so that the energy is equal to 1.  If all velocities are changed by the same multiplicative constant then the balls will follow the same trajectories at a different rate and hence there will be the same total number of collisions.

\begin{remark}
\label{re:zero}

(i) The problem of the number of collisions is invariant under time shifts.

(ii)
We recall \cite[Rem. 4.3]{BuD18b}. 
Let $\alpha(t)= \angle(\bx(t),\bv(t+))$. There is a unique $t_0\in\R$ such that $\alpha(t) >\pi/2$ for $t<t_0$, and $\alpha(t)<\pi/2$ for $t>t_0$. Right continuity yields $\alpha(t_0)\leq\pi/2$.  We have 
\begin{align}\label{j16.2}
|\bx(t) | \geq |\bx(t_0)|.
\end{align}
for all $t\in\R$.
\end{remark}

\begin{theorem}\label{s28.1} 
(\cite[Thm. 5.1]{BuD18b})
Recall assumptions (A1)-(A4) and time $t_0$ defined in Remark \ref{re:zero}.
 The family of $n$ balls can be partitioned into two non-empty subfamilies such that no ball from the first family collides with a ball in the second family in the time interval $\left[t_0+100 n^3|\bx(t_0)|,\infty\right)$. By symmetry and time reversal, a similar claim applies to $\left(-\infty,t_0-100 n^3|\bx(t_0)|\right]$.
\end{theorem}

\section{Proof of the main theorem}\label{se:proof}

\begin{proof}[Proof of Theorem \ref{j18.20}]
\emph{Step 1}. 
Consider a family of $n$ balls satisfying assumptions (A1)-(A4).
Recall $t_0 \in \R$ defined in Remark \ref{re:zero}.

Let  $[s_1,s_2]$ be the smallest interval containing $t_0$ satisfying the following two conditions.

(i) The balls  can be partitioned into two non-empty subfamilies such that no ball from the first family collides with a ball in the second family in the time interval  $(-\infty,s_1)$.

(ii)  The balls can be partitioned into two non-empty subfamilies such that no ball from the first family collides with a ball in the second family in the time interval  $(s_2,\infty)$.

The division into subfamilies in (i) and (ii) is not unique.
The subfamilies  in (i) need not be the same as those in (ii).

By \eqref{j16.2} and Theorem \ref{s28.1}, for any $t\in \R$,
\begin{align*}
s_2 - t_0 &\leq 100 n^3 |\bx(t_0)|
\leq 100 n^3 |\bx(t)|,\\
t_0-s_1 &\leq 100 n^3 |\bx(t_0)|
\leq 100 n^3 |\bx(t)|,
\end{align*}
so, for any $t\in \R$,
\begin{align}\label{j17.2}
s_2-s_1 
\leq  200 n^3 |\bx(t)|.
\end{align}

\medskip
\emph{Step 2}.
Fix a time interval $[u,u+1]$.
Since $|\bv(t+)|=1$ for all $t$, we have $|v^i(t)|\leq 1$ for all $i$ and $t$.  Hence, a ball can travel at most distance 1 in the time interval $[u,u+1]$. Suppose  the center of a ball $B_1$ is at $y$ at time $u$ and $B_1$ collides with another ball $B_2$ at a time $u_1\in[u,u+1]$. The distance from the center of $B_1$ to $y$ is at most 1 at time $u_1$ so the  distance from the center of $B_2$ to $y$ is at most 3 at the same time. It follows that distance from the center of $B_2$ to $y$ is at most 4 at time $u$. This implies that $B_2$ is a subset of the ball $\calB$ (not a billiard ball) centered at $y$ with radius 5 at time $u$. The volume of $\calB$  is $5^d$ times the volume of a ball with radius 1. Hence, $B_1$ might have collided with at most $5^d$ balls during  the  time interval $[u,u+1]$.
It follows that the number of pairs of balls that could have collided in $[u,u+1]$ is bounded by $5^d n/2$ (the factor $1/2$ is present so as not to count pairs twice). Note that $5^d n/2$ is the number of pairs of balls that could have collided in $[u,u+1]$, not the number of collisions, which could be much higher. We need an estimate of the number of pairs of balls because we want to use the results of \cite{BFK1}. In that paper, counting of collisions is based on the number of ``walls $B_i$'' (in the notation of that paper), which is  equal to the number of pairs of balls that could collide.

We will now apply an upper bound on the number of collisions given in \cite{BFK1}. In the notation of \cite{BFK1}, $r_{\text{max}}=1$ and $m_{\text{max}}=1$ because (i) we have assumed that the balls have radii 1, and
(ii) we have  assumed that  the masses of all balls are equal, so we can make them all equal to 1 without losing generality. 
In view of these remarks, the bound given on the left hand side of the last displayed formula on page 707 of \cite{BFK1} is 
$(8(2 n \sqrt{n}))^{n(n-1) -2}$. 
There seems to be a mistake here, in view of Remark 5.3 in \cite{BFK1}. The correct version should be $(8(2 n \sqrt{n} +2))^{n(n-1) -2}$.
Here $n(n-1) $ is twice the number of pairs of balls. 
Since only $5^dn/2$ pairs of balls can collide in $[u,u+1]$,
we replace $n(n-1) $ with $5^d n$ to see that the number of collisions of $n$ balls during an interval $[u,u+1]$ is bounded by 
\begin{align}\label{j17.1}
(8(2 n \sqrt{n}+2))^{5^d n -2}<(32 n^{3/2})^{5^d n -2}.
\end{align}
This bound agrees with Corollary 1.1 in \cite{BFK5} but our bound is more explicit.

We now offer a more formal justification of the bound in \eqref{j17.1}. Let $\calA$ be the family of all pairs of balls that collide in $[u,u+1]$. Consider a billiards evolution in which (i) pairs of balls in $\calA$ move along the same trajectories as in the original evolution in $[u,u+1]$, (ii) the trajectories of pairs of balls in $\calA$ are extended outside $[u,u+1]$ in the usual way, i.e., with elastic collisions, and (iii) pairs of balls that do not belong to $\calA$ do not collide, i.e., they pass through each other like ghosts. The results of \cite{BFK1} apply to this model with the number of ``walls $B_i$'' (in the notation of that paper)  equal to the number of pairs in $\calA$.

\medskip
\emph{Step 3}. 
By \eqref{j17.2} and \eqref{j17.1}, the number of collisions on the interval $[s_1,s_2]$ is bounded by 
\begin{align}\label{j17.3}
200 n^3 |\bx(t)| \left(32 n^{3/2}\right)^{5^d n -2},
\end{align}
for any $t\in \R$.
This bound is based on Theorem \ref{s28.1} proved under the assumptions (A1)-(A4). We will argue that \eqref{j17.3} holds even if (A3) and (A4) are not satisfied.

First, we will argue that \eqref{j17.3} holds even if (A4) does not hold.
Suppose that $\bx$ and $\bv$ do not satisfy  (A4).
Let $c_1=1/|\bv(0+)|$ and $\wt \bv(t+) = c_1 \bv(t+)$ for all $t$. Then
$|\wt \bv(t+)|=1$ for all $t$. We keep the same position at time $t_0$, i.e., $\wt \bx(t_0) = \bx(t_0)$. The balls will follow the same trajectories but at a different speed. Hence, $\inf_{t\in \R} |\bx(t)| = \inf_{t\in \R} |\wt \bx(t)| = |\wt \bx(t_0|$.
Let $[\wt s_1,\wt s_2]$ be defined as in Step 1 relative to $\wt \bx$. Then the number of collisions in $[s_1,s_2]$ in the system characterized by $\bx$ and $\bv$ is the same as the number of collisions in $[\wt s_1,\wt s_2]$ in the system characterized by $\wt\bx$ and $\wt\bv$. Since \eqref{j17.3} holds for the latter evolution, it also holds for the former.

Next we will argue that we do not need to assume  (A3) and (A4) for \eqref{j17.3} to hold.
Suppose that $\bx$ and $\bv$ do not necessarily satisfy (A3) and (A4), and recall $\wt\bx$ and $\wt\bv$ defined in the previous paragraph. For some $\bz\in\R^d$,
we have $\sum_{j=1}^n \wt x^j(t) =\sum_{j=1}^n \wt x^j(0) + t \bz $ for  all $t$.
Let $\bz_1(t) =  \frac 1 n \sum_{j=1}^n \wt x^j(0) + t \bz/n$,
$\wh x^j(t) = \wt x^j(t) - \bz_1( t)$
and $\wh v^j(t) = \wt v^j(t) - \bz/n$ for all $j=1,\dots,n$ and $t\in\R$.
The pair $\wh \bx$ and $\wh \bv$ is a representation of 
the dynamical system defined by $\wt \bx$ and $\wt \bv$ in a different inertial frame of reference. 
We have $\sum_{j=1}^n \wh x^j(t) =0$ for all $t$, so the functions $\wh \bx$ and $\wh \bv$ characterize an evolution satisfying (A3). This and the previous paragraph show that  \eqref{j17.3} holds with $\wh \bx$ in place of $\bx$.
A standard calculation shows that the function $a \to \sum_{j=1}^n |\wh x^j(t) - a \bz_1( t)|^2$ achieves the maximum at $a=0$ because $\sum_{j=1}^n \wh x^j(t) =0$. This implies that
 $|\wt \bx(t)| \geq |\wh \bx(t)|$ for every $t$, and, therefore, $\inf_{t\in \R} |\bx(t)| = \inf_{t\in \R} |\wt \bx(t)|\geq \inf_{t\in \R} |\wh \bx(t)|$. This completes the proof that \eqref{j17.3} holds even if (A3) and (A4) are not satisfied.

\medskip
\emph{Step 4}. 
Consider any subfamily $\calF $  of the balls. Suppose that $T_1(\calF) <T_2(\calF) $ are given and balls in $\calF$ do not collide with any balls outside of $\calF$ on the time interval $(T_1(\calF) ,T_2(\calF)) $. 
We will define $r(\calF),t_*(\calF), U_1(\calF)$ and $U_2(\calF) $. These  numbers depend not only on $\calF$, as indicated by the notation, but also on 
$T_1(\calF) $ and $T_2(\calF) $. Hopefully, our notation, chosen for typographical convenience, will not cause confusion. 

Let $n_\calF$ be the number of balls in $\calF$ and suppose that the indices of balls in $\calF$ are $i_1,\dots, i_{n_\calF}$.
Choose an inertial coordinate system $CS_\calF$ such that if $\wt x^{i_k}(t)$ is the position of the $i_k$-th ball $CS_\calF$ then $\sum_{k=1}^{n_\calF} \wt x^{i_k}(t)=0$ for all $t \in [T_1(\calF) ,T_2(\calF)] $.
Let
\begin{align}\notag
\bx_\calF(t) &= (\wt x^{i_1}(t), \dots, \wt x^{i_{n_\calF}}(t)),\\
r(\calF,t)& = \max_{1\leq j,k\leq n_\calF}
 |\wt x^{i_j}(t)-\wt x^{i_k}(t)|,\label{n11.1}\\
r(\calF) &= \inf_{t\in [T_1(\calF) ,T_2(\calF)]}
r(\calF,t),\label{n11.2}\\
t_*(\calF)&= \argmin_{t\in [T_1(\calF) ,T_2(\calF)]} r(\calF,t),\notag\\
|\bv_\calF|&= \left(\sum_{k=1}^{n_\calF} \left(v^{i_k}(t)\right)^2\right)^{1/2},
\qquad t\in [T_1(\calF) ,T_2(\calF)] .\notag
\end{align}

Since $\sum_{k=1}^{n_\calF} \wt x^{i_k}(t)=0$, the norm of the vector $\bx_\calF(t_*(\calF))$ is smaller in $CS_\calF$ than in any other coordinate system, for example, in a coordinate system with the origin at $\wt x^{i_1}(t_*(\calF))$. Hence,
\begin{align*}
 |\bx_\calF(t_*(\calF))|^2 \leq 
\sum_{k=2}^{n_\calF} |\wt x^{i_k}(t_*(\calF))-\wt x^{i_1}(t_*(\calF))|^2
\leq (n_\calF -1) r(\calF)^2,
\end{align*}
and, therefore,
\begin{align}\label{j18.12}
 |\bx_\calF(t_*(\calF))| \leq 
n_\calF^{1/2}  r(\calF).
\end{align}

Consider the following modified evolution of balls in $\calF$. Let the evolution of balls in $\calF$ remain as in the original system in the time interval $ [T_1(\calF) ,T_2(\calF)] $. Let the evolution continue before $T_1(\calF)$ and after $T_2(\calF) $, with balls in $\calF$ colliding according to the usual laws of elastic collisions, but with no collisions between balls in $\calF$ with balls outside $\calF$.
According to Remark \ref{re:zero} and \eqref{j16.2} there exists a unique $T_0(\calF) \in\R$ such that for all $t\in \R$,
\begin{align}\label{j17.4}
|\bx_\calF(t)| \geq |\bx_\calF(T_0(\calF))| .
\end{align}

Let  $[S_1(\calF),S_2(\calF)]$ be the smallest interval containing $T_0(\calF)$ satisfying the following two conditions.

(i) The family $\calF$  can be partitioned into two non-empty subfamilies $\calF_1$ and $\calF_2$ such that no ball in $\calF_1$ collides with a ball in $\calF_2$ in the time interval  $(-\infty,S_1(\calF))$.

(ii)  The family $\calF$ can be partitioned into two non-empty subfamilies $\calF_3$ and $\calF_4$ such that no ball in $\calF_3$ collides with a ball in $\calF_4$ in the time interval  $(S_2(\calF),\infty)$.

The division into subfamilies in (i) and (ii) is not unique. By Theorem \ref{s28.1}, $S_1(\calF)>-\infty$ and $S_2(\calF)<\infty$.
Let
\begin{align}\label{j17.6}
U_1(\calF) = \max(S_1(\calF), T_1(\calF)), 
\qquad U_2(\calF) = \min(S_2(\calF), T_2(\calF)).
\end{align}
It is possible that $U_1(\calF)=U_2(\calF) $.

We have $[U_1(\calF),U_2(\calF)]\subset[S_1(\calF),S_2(\calF)]$ so by
\eqref{j17.2} and rescaling by the speed $|\bv_\calF|$,
\begin{align}\label{j18.14}
U_2(\calF)-U_1(\calF) \leq 200 n_\calF^3 |\bx_\calF(t_*(\calF))|/|\bv_\calF|.
\end{align}

By \eqref{j17.3}, the number of collisions between balls in $\calF$ in the  time interval $[U_1(\calF),U_2(\calF)]$ is bounded by
\begin{align}\label{j17.5}
200 n_\calF^3 |\bx_\calF(t_*(\calF))| \left(32 n_\calF^{3/2}\right)^{5^d n_\calF -2}.
\end{align}

\medskip
\emph{Step 5}. 
We will construct a branching family $\calW$ with elements of  the form  
\begin{align*}
\Lambda(\calF) :=
(\calF, r(\calF), T_1(\calF) ,T_2(\calF)  , U_1(\calF) ,U_2(\calF) ),
\end{align*}
where $\calF$ is a subfamily of the balls.

Let $\calG_1$ be the set of all $n$ balls.
We initiate the construction of $\calW$ by declaring
$(\calG_1, r(\calG_1), -\infty , \infty,U_1(\calG_1),U_2(\calG_1))$ to be the only ``individual'' in the first generation of the branching structure  $\calW$. In other words, $T_1(\calG_1)=-\infty$, $T_2(\calG_1)=\infty$, and $U_1(\calG_1)$ and $U_2(\calG_1)$ are defined as in \eqref{j17.6} with $\calG_1$ in place of $\calF$.

We will now describe the branching mechanism. Suppose that
\begin{align}\label{j17.7}
(\calF, r(\calF), T_1(\calF) ,T_2(\calF)  , U_1(\calF) ,U_2(\calF) )
\in \calW.
\end{align}

We always have $ [U_1(\calF),U_2(\calF)] \subset [T_1(\calF),T_2(\calF)]$.
If $ [U_1(\calF),U_2(\calF)] = [T_1(\calF),T_2(\calF)]$ then we declare 
the sextuplet in \eqref{j17.7} to be a leaf of the branching tree, i.e., this sextuplet has no offspring. We also declare 
the sextuplet in \eqref{j17.7} to be a leaf if $n_\calF\leq 2$.

Suppose that $ [U_1(\calF),U_2(\calF)] \neq [T_1(\calF),T_2(\calF)]$ and $n_\calF\geq 3$.
Recall families $\calF_1, \calF_2, \calF_3$ and $\calF_4$ defined in conditions (i) and (ii) below \eqref{j17.4}.
Let 
\begin{align*}
T_1(\calF_1) &=T_1(\calF_2) = T_1(\calF),\qquad
T_2(\calF_1) =T_2(\calF_2) = U_1(\calF),\\
T_1(\calF_3) &=T_1(\calF_4) = U_2(\calF),\qquad
T_2(\calF_3) =T_2(\calF_4) = T_2(\calF).
\end{align*}
We declare the following four sextuplets to be (some of the) offspring of the sextuplet in \eqref{j17.7},
\begin{align}\label{j18.1}
&(\calF_k, r(\calF_k), T_1(\calF_k) ,T_2(\calF_k)  , U_1(\calF_k) ,U_2(\calF_k) ),\qquad k=1,2,3,4.
\end{align}
Here $ U_1(\calF_k)$ and $U_2(\calF_k) $ are defined as in \eqref{j17.6} with $\calF_k$ in place of $\calF$.

It is easy to check that if $\calF_5 = \calF$, $T_1(\calF_5)= U_1(\calF)$ and $T_2(\calF_5)= U_2(\calF)$ then $r(\calF_5)\leq r(\calF)$, $U_1(\calF_5)= U_1(\calF)$ and $U_2(\calF_5)= U_2(\calF)$.

If $r(\calF) \leq 4 n_\calF$ then we declare that the sextuplet in \eqref{j17.7} has five offspring---the four offspring listed in \eqref{j18.1} and 
\begin{align}\label{j18.5}
&(\calF_5, r(\calF_5), U_1(\calF) ,U_2(\calF)  , U_1(\calF) ,U_2(\calF) ).
\end{align}
This case is illustrated in Fig.~\ref{fig1}.
\begin{figure} \includegraphics[width=0.9\linewidth]{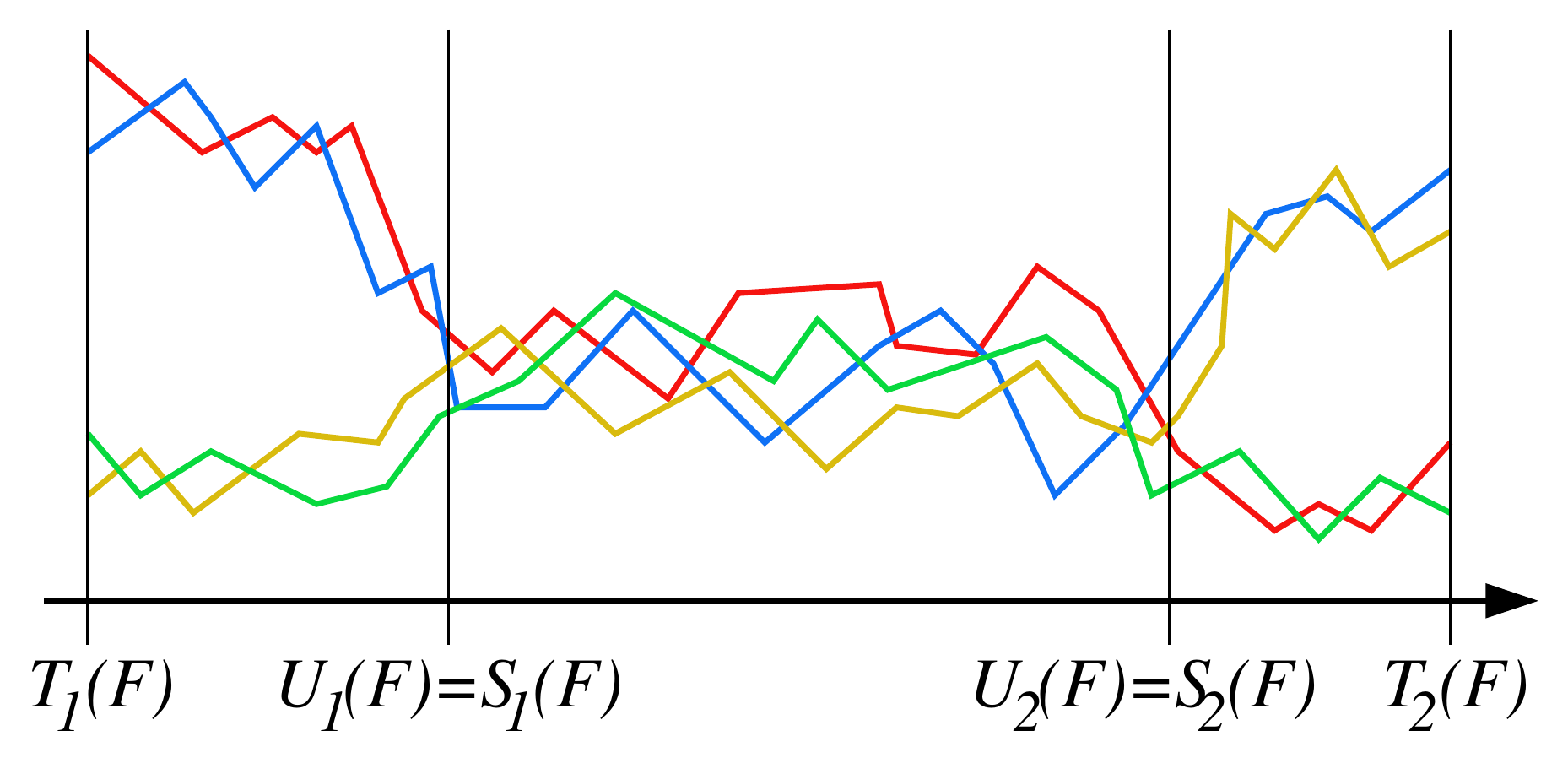}
\caption{
 Schematic drawing of a branching event with five offspring. The parent family $\calF$ is represented by all trajectories on the interval  $(T_1(\calF) ,T_2(\calF))$. In this generic case, $(U_1(\calF) ,U_2(\calF))$ lies strictly inside $(T_1(\calF) ,T_2(\calF))$. All trajectories are close to each other (at least at one time) in $(U_1(\calF) ,U_2(\calF))$. The five offspring consist of two non-interacting families to the left of $U_1(\calF)=S_1(\calF)$, two non-interacting families to the right of $U_2(\calF)=S_2(\calF)$, and the original family $\calF$ restricted to  the interval $(U_1(\calF) ,U_2(\calF))$. The latter offspring will not have any descendants---it is a leaf in the branching structure.
}
\label{fig1}
\end{figure}

Next we will discuss the case when $r(\calF) > 4 n_\calF$. In this case, all sextuplets listed  in \eqref{j18.1} will be declared to be offspring of the sextuplet in \eqref{j17.7} but there will be  more offspring constructed as follows.

Let $t_1= U_1(\calF)$ and  
\begin{align}\label{j18.10}
\beta = (r(\calF) - 2 n_\calF)/(n_\calF -1).
\end{align}
We will argue that $\calF$ can be partitioned into nonempty disjoint families $\calH^1_1$ and $\calH^1_2$ such that the distance between any ball in $\calH^1_1$ and  any ball in $\calH^1_2$ is greater than $\beta$ at time $t_1$. If this is not the case then every two balls in $\calF$ are connected by a chain of balls with distances between consecutive balls less than or equal to $\beta$. Hence, the distance between the centers of endpoint balls in the chain is bounded by 
\begin{align*}
(n_\calF-1) (\beta + 2)= r(\calF) - 2 n_\calF + 2 (n_\calF-1) = r(\calF) -2.
\end{align*}
This contradicts the definitions 
\eqref{n11.1}-\eqref{n11.2} of $r(\calF,t)$ and $r(\calF)$ because according to these definitions, there must exist balls whose centers are at a distance equal to or greater than $r(\calF)$ for every $t\in [T_1(\calF) ,T_2(\calF)]$.
We conclude that families $\calH^1_1$ and $\calH^1_2$ exist.

The velocities of  balls in $\calF$ are bounded by $|\bv_\calF|$ so no ball in $\calH^1_1$ can collide with any ball in $\calH^1_2$ in the time interval $[t_1,  t_1 + \beta/|\bv_\calF|]$.

Let 
\begin{align}\label{j18.11}
k_* &= \lceil(U_2(\calF) - U_1(\calF))|\bv_\calF|/\beta\rceil , \\
t_k &= t_1 + (k-1)\beta/|\bv_\calF|, \qquad k =2,\dots,k_* ,\notag \\
t_{k_*+1} &= U_2(\calF).\notag
\end{align}
For every $k=2,\dots, k_*$, we can find nonempty disjoint families $\calH^k_1$ and $\calH^k_2$ such that the distance between any ball in $\calH^k_1$ and  any ball in $\calH^k_2$ is greater than $\beta$ at time $t_k$. No ball in $\calH^k_1$ can collide with any ball in $\calH^k_2$ in the time interval $[t_k,t_{k+1}]$.

We declare the following sextuplets to be offspring of the sextuplet in \eqref{j17.7},
\begin{align}\label{j18.6}
(\calH^k_i, r(\calH^k_i), t_k ,t_{k+1} , U_1(\calH^k_i) ,U_2(\calH^k_i) ),
\end{align}
for $k=1,\dots, k_*$ and $i=1,2$.
Hence,
in the case when $r(\calF) > 4 n_\calF$,  sextuplets listed  in \eqref{j18.1} and \eqref{j18.6} are offspring of the sextuplet in \eqref{j17.7}. This case is illustrated in Fig.~\ref{fig2}.
\begin{figure} \includegraphics[width=0.9\linewidth]{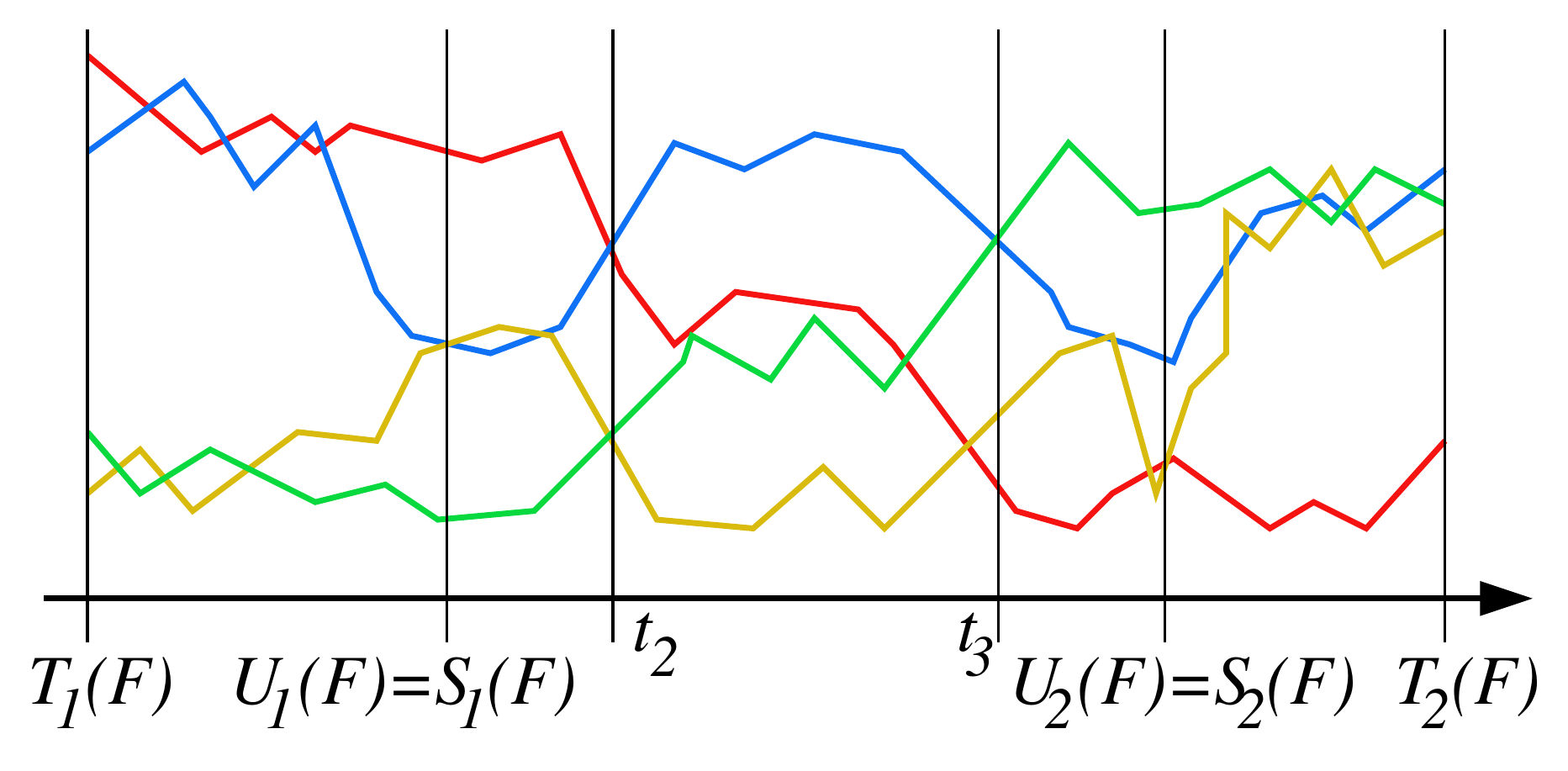}
\caption{
 Schematic drawing of a branching event with more than five offspring. The parent family $\calF$ is represented by all trajectories on the interval  $(T_1(\calF) ,T_2(\calF))$. In this generic case, $(U_1(\calF) ,U_2(\calF))$ lies strictly inside $(T_1(\calF) ,T_2(\calF))$. There is no time  in $(U_1(\calF) ,U_2(\calF))$ such that all trajectories are close to each other. There are ten offspring. Four of these consist of two non-interacting families to the left of $U_1(\calF)=S_1(\calF)$ and two non-interacting families to the right of $U_2(\calF)=S_2(\calF)$. On each of the intervals  $(U_1(\calF) ,t_2)$, $(t_2,t_3)$ and $(t_3 ,U_2(\calF))$ there are at least two non-interacting families of trajectories. On each of these intervals, two non-interacting families are chosen and declared to be offspring of $\calF$.
}
\label{fig2}
\end{figure}

\medskip
\emph{Step 6}.
We will estimate some quantities characterizing  $\calW$.
We will write $\Lambda(\calG) \prec \Lambda(\calF)$ to indicate that $\Lambda(\calG)$ is an offspring of $\Lambda(\calF)$.
If $\Lambda(\calG) , \Lambda(\calF)\in\calW$ and $\Lambda(\calG) \prec \Lambda(\calF)$ then either $\Lambda(\calG) $ is a leaf or $n_\calG<n_\calF$. It follows that the number of generations in $\calW$ is bounded by $n$.

An individual in $\calW$ has five offspring 
in the case $r(\calF) \leq 4 n_\calF$.

For the next calculation, recall that $n_\calF \geq 3$.
If $r(\calF) > 4 n_\calF$ then, in view of \eqref{j18.12}, \eqref{j18.14}, \eqref{j18.10} and \eqref{j18.11}, the number of offspring is 
bounded above by
\begin{align}
4+2k_* &= 4+ 2  \lceil(U_2(\calF) - U_1(\calF))|\bv_\calF|/\beta\rceil
\leq 6+ 2  (U_2(\calF) - U_1(\calF))|\bv_\calF|/\beta \notag \\
&= 6+ 2  \frac{(U_2(\calF) - U_1(\calF))|\bv_\calF|}
{(r(\calF) - 2 n_\calF)/(n_\calF -1)}
\leq 6+ 2  \frac{(U_2(\calF) - U_1(\calF))|\bv_\calF|}
{(r(\calF) / 2 )/(n_\calF -1)} \notag \\
&\leq 6+ 2  \frac{200 n_\calF^3 |\bx_\calF(t_*(\calF))|}
{(r(\calF) / 2 )/(n_\calF -1)}
\leq 6+ 2  \frac{200 n_\calF^3 n_\calF ^{1/2}r(\calF)}
{(r(\calF) / 2 )/(n_\calF -1)} \notag \\
&\leq 6+ 800n_\calF^{9/2} \leq 1000 n_\calF ^{9/2}
\leq 1000 n^{9/2}.\label{j20.1}
\end{align}
This upper bound holds also in the case $r(\calF) \leq 4 n_\calF$.
Thus \eqref{j20.1} implies that
the number of individuals in the $k$-th generation is bounded by $(1300 n^3)^{k-1} $. Since the number of generations is bounded by $n$, the total number of individuals in $\calW$ is bounded by 
\begin{align}\label{j20.2}
n (1000 n^{9/2})^{n-1} \leq 1000^n n^{9n/2}.
\end{align}
 
\medskip
\emph{Step 7}. 
We will now bound the number of collisions. It follows from Step 5 that 
if two balls collide then there must exist a leaf
\begin{align*}
\Lambda(\calF)=(\calF, r(\calF), T_1(\calF) ,T_2(\calF)  , U_1(\calF) ,U_2(\calF) )
\in \calW
\end{align*}
such that the two balls belong to $\calF$ and the collision takes place in the interval $[T_1(\calF), T_2(\calF)]$.

First we will count collisions in open intervals of the form $(T_1(\calF), T_2(\calF))$.

If $n_\calF \leq 2$ then the number of collisions in $(T_1(\calF), T_2(\calF))$ is bounded by 1.

The argument in Step 5 (see \eqref{j18.5}) shows that
if $n_\calF \geq 3$ and  $\Lambda(\calF)$ is a leaf then $(T_1(\calF), T_2(\calF))=(U_1(\calF), U_2(\calF))$ and $r(\calF)\leq 4 n_\calF\leq 4 n$.
Hence, we can use \eqref{j17.5} as an upper bound for the number of collisions in $(T_1(\calF), T_2(\calF))$. We combine \eqref{j18.12}, \eqref{j17.5} and the estimate $r(\calF)\leq 4 n$ to obtain the following bound on the number of collisions in $(T_1(\calF), T_2(\calF))$ associated with $\Lambda(\calF)$,
\begin{align}\label{j20.3}
&200 n_\calF^3 |\bx_\calF(t_*(\calF))| \left(32 n_\calF^{3/2}\right)^{5^d n_\calF -2}
\leq 
200 n_\calF^3 n_\calF^{1/2} r(\calF) \left(32 n_\calF^{3/2}\right)^{5^d n_\calF -2}\\
&\qquad\leq
200 n^3 n^{1/2}\cdot 4n \left(32 n^{3/2}\right)^{5^d n -2}
=
800 n^{9/2} \left(32 n^{3/2}\right)^{5^d n -2}.\notag
\end{align}
This upper bound applies also to leaves $\Lambda(\calF)$ with $n_\calF \leq 2$.

The number of leaves in $\calW$ is bounded by the quantity in \eqref{j20.2} so, in view of \eqref{j20.3}, the total number of collisions in open intervals of the form $(T_1(\calF), T_2(\calF))$ is bounded by
\begin{align}\notag
&1000^n n^{9n/2}
800 n^{9/2} \left(32 n^{3/2}\right)^{5^d n -2}
=800\cdot 1000^n n^{(9/2)(n+1)} \left(32 n^{3/2}\right)^{5^d n -2}\\
&\quad =
800 \left( 1000 \cdot 32^{5^d}\right)^n  n^{((3/2)5^d+9/2) n +3/2}.
\label{n11.3}
\end{align}

The number of collisions at times $T_1(\calF)$ or $ T_2(\calF)$ is bounded by the product of (i) the number of individuals in $\calW$, (ii) number of endpoints of an interval, and (iii) one half of the  number of balls, so, in view of \eqref{j20.2}, it is bounded by
\begin{align*}
1000^n n^{9n/2} \cdot 2 \cdot n/2=1000^n n^{9n/2+1}.
\end{align*}
We combine this bound with \eqref{n11.3} to conclude that the number of collisions is bounded by
\begin{align*}
&800 \left( 1000 \cdot 32^{5^d}\right)^n  n^{((3/2)5^d+9/2) n +3/2}
+ 1000^n n^{9n/2+1}\\
&\quad\leq 
1600 \left( 1000 \cdot 32^{5^d}\right)^n  n^{((3/2)5^d+9/2) n +3/2}.
\end{align*}

\end{proof}

\begin{remark}
(i)
The estimates in Step 6 are crude and can be easily improved but we do not see a way to reduce the quantity in \eqref{j20.2} so that its logarithm is $o(n\log n)$. Even if we could, the logarithm of the quantity in \eqref{j17.3} is not 
$o(n\log n)$ so the bound in  \eqref{j19.2} would not change in a significant way.

(ii)
Let $\tau_d$ denote
the  kissing number  of a $d$-dimensional ball, i.e.,  the maximum
number of mutually nonoverlapping translates of the ball that can be arranged
so that they all touch the ball.
According to  \cite[Thm.~1.1.3]{Bez}, 
\begin{align*}
2^{0.2075d(1+ o(1))} \leq \tau_d \leq 2^{0.401d(1+o(1))}.
\end{align*}
In Step 2, we derived the bound $5^d$ for the number of balls that could collide with a given ball on a time interval of length one. The lower bound for the kissing number shows that the bound $5^d$ cannot be improved to be less than exponential in $d$.
\end{remark}

\section{Acknowledgments}

I am grateful to Soumik Pal for very helpful advice. I thank the anonymous referees for many significant suggestions for improvement.

\bibliographystyle{alpha}
\def\cprime{$'$}

\end{document}